\documentclass[a4paper,10pt]{article}

\usepackage[T1]{fontenc}
\usepackage{lmodern}
\usepackage[utf8]{inputenc}
\usepackage{amsmath,amssymb,amsthm}
\usepackage[left=3.5cm,right=3.5cm,top=2cm,bottom=2cm]{geometry}
\usepackage{tikz}
\usepackage{float}
\usepackage{xcolor}
\usepackage{hyperref}
\hypersetup{hidelinks}
\newcommand{\bref}[1]{\textbf{\ref{#1}}}
\newcommand{\la}{\operatorname{la}}
\newcommand{\latop}{\operatorname{la}_{\mathrm{top}}}
\newcommand{\Rays}{\mathcal{R}}

\newtheorem{theorem}{Theorem}[section]
\newtheorem{conjecture}{Conjecture}[section]
\newtheorem{proposition}{Proposition}[section]
\newtheorem{lemma}{Lemma}[section]
\newtheorem{corollary}{Corollary}[section]
\newtheorem{claim}{Claim}[section]
\newtheorem{question}{Question}[section]

\theoremstyle{definition}
\newtheorem{definition}{Definition}[section]

\newtheoremstyle{theoremstar}
{} % Space above
{} % Space below
{\itshape} % Body font
{} % Indent
{\bfseries} % Head font
{.} % Punctuation
{.5em} % Space after head
{\thmname{#1}\thmnote{ \bfseries #3}}
\makeatother

\theoremstyle{theoremstar}
\newtheorem*{theoremstar}{Theorem}

\title{Linear arboricity conjecture for infinite graphs}
\author{Leandro Aurichi, Rodrigo Monteiro and Caio Rodrigues}

\newcommand{\Addresses}{{
		\bigskip
		\footnotesize
		
		L.~Aurichi, Instituto de Ci\^encias Matem\'aticas e de Computa\c c\~ao, Universidade de S\~ao Paulo\\
		Avenida Trabalhador s\~ao-carlense, 400, S\~ao Carlos, SP, 13566-590, Brazil\par\nopagebreak
		\textit{E-mail}: \texttt{aurichi@icmc.usp.br}
	
		\medskip
		R.~Monteiro, Instituto de Ci\^encias Matem\'aticas e de Computa\c c\~ao, Universidade de S\~ao Paulo\\
		Avenida Trabalhador s\~ao-carlense, 400, S\~ao Carlos, SP, 13566-590, Brazil\par\nopagebreak
		\textit{E-mail}: \texttt{rodrigosm@usp.br}
		
			C.~Rodrigues, Instituto de Ci\^encias Matem\'aticas e de Computa\c c\~ao, Universidade de S\~ao Paulo\\
		Avenida Trabalhador s\~ao-carlense, 400, S\~ao Carlos, SP, 13566-590, Brazil\par\nopagebreak
		\textit{E-mail}: \texttt{caiofernandorodrigues@usp.br}
}}
\date{}

\begin{document}

\maketitle

\begin{abstract}
The linear arboricity $\la(G)$ of a graph $G$ is the least cardinality of linear forests, that is, forests of maximum degree at most $2$, into which its edge set $E(G)$ can be decomposed. The Linear Arboricity Conjecture asserts that $\la(G)\leq\lceil(\Delta(G)+1)/2\rceil$ for every finite graph $G$, where $\Delta(G)$ denotes the maximum degree of $G$. We extend this conjecture to infinite graphs of finite maximum degree and prove that its finite and infinite versions are equivalent. We introduce topological linear arboricity $\latop (G)$ by requiring the linear forests to contain no topological circle, and show that it differs from linear arboricity by at most one. Finally, we prove that every $2k$-regular graph of girth at least $2k$ has topological linear arboricity at most $k+1$.
\end{abstract}



\section{Introduction}

 All graphs in this paper are simple, but they are not assumed to be finite. An edge decomposition of a graph is a family of pairwise edge-disjoint subgraphs whose edge sets partition the edge set of the graph.

\begin{definition}
A \emph{ray} is a one-way infinite path, and a \emph{double ray} is a two-way infinite path. We write $\Rays(G)$ for the set of rays of a graph $G$. A \emph{linear forest} is a forest of maximum degree at most $2$. Thus, in an infinite graph, its components are finite paths, rays, or double rays. The \emph{linear arboricity} of $G$, denoted by $\la(G)$, is the least cardinality of an edge decomposition of $G$ into linear forests.
\end{definition}

Linear arboricity was introduced by Harary~\cite{harary1970} for finite graphs. Since every matching is a linear forest and every linear forest is a forest, the parameter lies between arboricity and the edge-chromatic number, linking the decomposition problems addressed by the classical theorems of Nash--Williams and Vizing~\cite{nashwilliams1964,vizing1964}. Linear arboricity can also be viewed as a relaxation of decomposition into Hamilton paths, which are linear forests with the maximum possible number of edges on a given vertex set, see~\cite{christoph2025} for this perspective and recent developments.

For a graph $G$ of finite maximum degree $\Delta(G)$, each linear forest contains at most two edges incident with any given vertex. Consequently,
$
\la(G)\geq \left\lceil\frac{\Delta(G)}{2}\right\rceil.
$
When $G$ is a finite $\Delta(G)$-regular graph, a simple edge-counting argument yields the stronger bound
$
\la(G)\geq \left\lceil\frac{\Delta(G)+1}{2}\right\rceil.
$
Akiyama, Exoo and Harary conjectured that these lower bounds are best possible~\cite{akiyama1981}.

\begin{conjecture}[Linear Arboricity Conjecture~\cite{akiyama1981}]
	\label{conj}
For every finite graph $G$,
\[
\left\lceil\frac{\Delta(G)}{2}\right\rceil\leq \la(G)\leq \left\lceil\frac{\Delta(G)+1}{2}\right\rceil
\]
\end{conjecture}

This conjecture is known to be true when $\Delta(G)\in\{3,4,5,6,8,10\}$,  see~\cite{akiyama1980,akiyama1981,enomoto1984,guldan1986}. It is also known for several structural classes of graphs, see, for example,~\cite{basavaraju2020,gao2024}. Alon proved that the conjecture is asymptotically correct~\cite{alon1988}. The error term was later improved in~\cite{ferber2020,lang2023}, and a recent bound of $\Delta(G)/2+O(\log|V(G)|)$ was obtained in~\cite{christoph2025}. To the best of our knowledge, however, linear arboricity had not previously been considered for infinite graphs. We propose the following extension.

\begin{conjecture}
	\label{arbinf}
For every infinite graph $G$ with finite maximum degree $\Delta(G)$,
\[
\left\lceil\frac{\Delta(G)}{2}\right\rceil\leq \la(G)\leq \left\lceil\frac{\Delta(G)+1}{2}\right\rceil.
\]
\end{conjecture}

The main purpose of Section \ref{secarb} is to show that passing to infinite graphs does not lead to a genuinely different conjecture. More precisely, we prove the following equivalence.

\begin{theoremstar}[\bref{equiv}]
	Conjectures \ref{conj} and \ref{arbinf} are equivalent.
\end{theoremstar}

In the same section, we also study extensions of prescribed linear forests of finite graphs to Hamiltonian decompositions. Motivated by a conjecture of Javadi and Miralaei~\cite{javadi2026} for finite complete graphs, we prove the following result for the countable infinite complete graph $K_\omega$. Here, a Hamiltonian double ray is a double ray containing every vertex of the graph.

\begin{theoremstar}[\bref{hamil}]
	Let $H$ be a finite subgraph of $K_\omega$ whose edges are colored with finitely many colors in such a way that every color class induces a linear forest. Then this coloring extends to an edge-coloring of $K_\omega$ with countably many colors such that every color class induces a Hamiltonian double ray.
\end{theoremstar}

The second part of the paper turns to a phenomenon specific to infinite graphs. We begin by recalling the necessary notions concerning ends and their topology.

\begin{definition}
Let $G$ be a graph. A \emph{tail} of a ray $r$ is an infinite connected subgraph of $r$. Define a relation $\sim$ on $\Rays(G)$ by declaring two rays $r,s\in\Rays(G)$ equivalent if, for every finite set $S\subseteq V(G)$, some tails of $r$ and $s$ lie in the same component of $G-S$. An \emph{end} of $G$ is an equivalence class of this relation, and we write
$
\Omega(G)=\Rays(G)/{\sim}
$
for the set of ends of $G$.

\end{definition}

Regard every edge of $G$ as a copy of the unit interval $[0,1]$ and let $|G|=G\cup\Omega(G)$. For a finite set $S\subseteq V(G)$ and $[r]\in\Omega(G)$, let $C(S,[r])$ be the component of $G-S$ containing a tail of $r$. A basic neighborhood of $[r]$ consists of $C(S,[r])$, all ends having a ray with a tail in this component, and half-open intervals on the edges from $S$ to $C(S,[r])$. Together with the open sets of the $1$-complex of $G$, these sets define the end topology on $|G|$. When $G$ is connected and locally finite, this space is the Freudenthal compactification of $G$.

Ends were introduced by Freudenthal in his study of topological spaces and groups~\cite{freudenthal1931}, and Halin later gave their graph-theoretic formulation in terms of equivalence classes of rays~\cite{halin1964}.
 The topology of $|G|$ allows infinite subgraphs to exhibit phenomena that do not occur in the finite setting. In particular, a linear forest may contain double rays whose closure in $|G|$ is homeomorphic to the circle $\mathbb{S}^1$, with its usual topology. Such a subspace is called a \emph{topological circle}, and it is \emph{Hamiltonian} if it contains every vertex of $G$. This motivates the following topological variant of linear arboricity.

\begin{definition}
	\label{toparb}
For a subgraph $L\subseteq G$, let
$$
\Omega_G(L)=\{[r]\in\Omega(G): L\text{ contains a ray belonging to }[r]\},
$$
and endow $|L|=L\cup\Omega_G(L)$ with the subspace topology inherited from $|G|$. A \emph{topological linear forest} of $G$ is a linear forest $L$ such that $|L|$ contains no topological circle. The \emph{topological linear arboricity} of $G$, denoted by $\latop(G)$, is the least cardinality of an edge decomposition of $G$ into topological linear forests.
\end{definition}

We next compare linear arboricity with  topological linear arboricity. Although a linear forest may give rise to topological circles in $|G|$, this phenomenon has only a limited effect on the number of forests required: as we show in Section \ref{sectop}, the two parameters differ by at most one.

\begin{theoremstar}[\bref{count}]
	 Let $G$ be a graph. If $G$ has finite maximum degree $\Delta(G)$, then $\la(G)\leq\latop(G)\leq\la(G)+1$. If $\la(G)$ is infinite, then $\latop(G)=\la(G)$.
\end{theoremstar}

We also consider even-regular graphs of large girth. Mishra proved that every finite $2k$-regular graph of girth at least $2k$ has linear arboricity at most $k+1$~\cite[Theorem 3]{mishra2025}. In Section \ref{sectop}, we extend this bound to the topological linear arboricity.

\begin{theoremstar}[\bref{girth}]
	Let $G$ be a $2k$-regular graph of girth at least $2k$. Then $\latop(G)\leq k+1$.
\end{theoremstar}

We  conclude the paper with an open question on the optimal general bound for topological linear arboricity, together with a problem concerning decompositions of cubic graphs into isomorphic linear forests.

\section{Linear arboricity}

\label{secarb}
The assumption that $\Delta(G)$ is finite in Conjecture \ref{arbinf} is essential. Indeed, every graph with unbounded vertex degrees has infinite linear arboricity. Otherwise, if $E(G)$ could be decomposed into finitely many linear forests $F_1,\ldots,F_k$, then each vertex would have degree at most $2k$, since $d_{F_i}(v)\leq2$ for every $i$. Thus $\Delta(G)\leq2k$, a contradiction. In particular, local finiteness alone does not guarantee finite linear arboricity.

The tool used to pass from Conjecture \ref{conj} to Conjecture \ref{arbinf} is the notion of  $F$-limit.

\begin{definition}
A filter $F$ on $\mathbb{N}$ is a nonempty family of subsets of $\mathbb{N}$ such that
\begin{enumerate}
\item $\varnothing\notin F$;
\item if $A\in F$ and $A\subseteq B\subseteq\mathbb{N}$, then $B\in F$;
\item if $A,B\in F$, then $A\cap B\in F$.
\end{enumerate}
A filter is an \emph{ultrafilter} if it is maximal under inclusion, and it is \emph{non-principal} if it contains no finite set.
\end{definition}

We will use the following basic properties of an ultrafilter $F$ on $\mathbb{N}$. For every $A\subseteq\mathbb{N}$, exactly one of $A$ and $\mathbb{N}\setminus A$ belongs to $F$. Consequently, if a finite union of sets belongs to $F$, then at least one of those sets belongs to $F$. If $F$ is non-principal, it contains every cofinite subset of $\mathbb{N}$ and no finite subset. In particular, every finite intersection of members of $F$ belongs to $F$ and is infinite.
Fix a non-principal ultrafilter $F$ on $\mathbb{N}$ for the remainder of the paper.

\begin{definition}
Let $G$ be a graph and let $\langle H_n\rangle_n$ be a sequence of subgraphs of $G$. Its $F$-limit is the subgraph $H\subseteq G$ with
\[
V(H)=\{v\in V(G):\{n\in\mathbb{N}:v\in V(H_n)\}\in F\}
\]
and
\[
E(H)=\{e\in E(G):\{n\in\mathbb{N}:e\in E(H_n)\}\in F\}.
\]
\end{definition}

Informally, the $F$-limit consists of those vertices and edges that occur in sufficiently many members of the sequence, where the ultrafilter determines the meaning of ``sufficiently many'' and ensures compatibility among finitely many such choices. This approach was first used in the context of infinite graphs in~\cite{aurichi2026limits} and has since been applied to problems concerning cycle covers and orientations~\cite{aurichi2025cycle,aurichi2026orientations}.
 In our setting, each graph in the sequence comes with a decomposition into linear forests, and taking the $F$-limit produces coherent limiting color classes.

\begin{theorem}
	\label{equiv}
Conjectures \ref{conj} and \ref{arbinf} are equivalent.
\end{theorem}

\begin{proof}
Assume first that Conjecture \ref{conj} holds, and let $G$ be an infinite graph of finite maximum degree $\Delta(G)$. Since every connected component of $G$ is countable and the components can be treated independently, we may assume that $G$ is connected. Fix an enumeration $V(G)=\{v_0,v_1,\ldots\}$ and let $G_n=G[\{v_0,\ldots,v_n\}]$. Then $G$ is the $F$-limit of $\langle G_n\rangle_n$. Indeed, every vertex and every edge of $G$ is absent from only finitely many of the graphs $G_n$, and therefore belongs to the $F$-limit.

Let
$
k=\left\lceil\frac{\Delta(G)+1}{2}\right\rceil.$
Since $\Delta(G_n)\leq \Delta(G)$ for every $n\in\mathbb{N}$, Conjecture \ref{conj} gives
$
\la(G_n)\leq \left\lceil\frac{\Delta(G_n)+1}{2}\right\rceil
\leq \left\lceil\frac{\Delta(G)+1}{2}\right\rceil=k.
$
Hence each $G_n$ admits an edge decomposition into at most $k$ linear forests. After adding empty color classes when necessary, write these forests as $F_1^n,\ldots,F_k^n$, and for $1\leq i\leq k$ let $F_i$ be the $F$-limit of $\langle F_i^n\rangle_n$.

Each $F_i$ is a linear forest. Indeed, if $F_i$ contained a finite cycle, then, by the finite-intersection property of $F$, the set of indices $n$ for which all edges of this cycle belong to $F_i^n$ would belong to $F$, contradicting the acyclicity of $F_i^n$. The same argument rules out vertices of degree at least $3$ in $F_i$ and shows that $F_i$ and $F_j$ are edge-disjoint whenever $i\neq j$.

It remains to check that these forests cover $E(G)$. Fix $e\in E(G)$, and let $A=\{n\in\mathbb{N}:e\in E(G_n)\}$. The set $A$ is cofinite and hence belongs to $F$. For $1\leq i\leq k$, let $A_i=\{n\in A:e\in E(F_i^n)\}$. Since the $F_i^n$ form an edge decomposition of $G_n$, the sets $A_1,\ldots,A_k$ partition $A$. As $F$ is an ultrafilter and $A\in F$, one of the sets $A_i$ belongs to $F$. Hence $e\in E(F_i)$. Therefore the forests $F_1,\ldots,F_k$ cover $E(G)$, and thus $\la(G)\leq k$, proving Conjecture \ref{arbinf}.

Conversely, assume Conjecture \ref{arbinf} and let $G$ be finite. Let $H$ be the graph whose connected components are countably many disjoint copies of $G$. Then $H$ is infinite and $\Delta(H)=\Delta(G)$. By Conjecture \ref{arbinf}, $\la(H)\leq\lceil(\Delta(G)+1)/2\rceil$. Restricting such a decomposition of $H$ to any one component gives the same upper bound for $\la(G)$. Hence Conjecture \ref{conj} holds.
\end{proof}

We begin by pointing out a first difference between linear arboricity for finite and infinite graphs. Classical results of Akiyama, Exoo, and Harary~\cite{akiyama1981} and of Enomoto and Péroche~\cite{enomoto1984} show that every finite $4$-regular graph has linear arboricity $3$, while every finite $6$-regular graph has linear arboricity $4$. The situation changes for infinite graphs. The infinite $4$-regular and $6$-regular trees have linear arboricity $2$ and $3$, respectively, attaining the lower bound $\lceil\Delta(G)/2\rceil$ rather than the finite regular bound $\lceil(\Delta(G)+1)/2\rceil$. Thus, although the finite and infinite versions of the Linear Arboricity Conjecture are equivalent, the behaviour of linear arboricity itself may be genuinely different in the infinite setting.

In the finite setting, linear arboricity is closely related to Hamiltonian decompositions of complete graphs. Recall that a Hamilton cycle contains every vertex of the graph, and that a Hamiltonian decomposition is an edge decomposition into Hamilton cycles. Hilton gave a characterization of those decompositions of a complete graph into linear forests that can be extended to a Hamiltonian decomposition of a larger complete graph~\cite{hilton1984}; further extension problems in this direction were studied in~\cite{liu2023}. Motivated by this connection, Javadi and Miralaei proposed the following strengthening.

\begin{conjecture}[Javadi--Miralaei~\cite{javadi2026}]
	\label{conj2}
Let $H$ be a subgraph of $K_{2n+1}$ with $n$ edges, and suppose that its edges are colored so that every color class induces a linear forest. Then the coloring extends to an edge-coloring of $K_{2n+1}$ with $n$ colors in which every color class induces a Hamilton cycle.
\end{conjecture}

Conjecture \ref{conj2} asks whether prescribed linear forests can be placed inside distinct members of a Hamiltonian decomposition. For the countable infinite complete graph $K_\omega$, the corresponding objects of the decomposition are spanning double rays. In the topology of $|K_\omega|$, these double rays play the role of Hamiltonian topological circles, that is, topological circles containing every vertex and end. The following countable analogue of Conjecture \ref{conj2} holds.

\begin{theorem}
	\label{hamil}
Let $H$ be a finite subgraph of $K_\omega$ whose edges are colored with finitely many colors in such a way that every color class induces a linear forest. Then this coloring extends to an edge-coloring of $K_\omega$ with countably many colors such that every color class induces a Hamiltonian double ray.
\end{theorem}

\begin{proof}
We first prove an auxiliary claim.

\begin{claim}
	\label{claim1}
	Let $G$ be a countable infinite graph such that every vertex has only finitely many non-neighbours. Then every finite linear forest $L\subseteq G$ is contained in a spanning double ray of $G$.
\end{claim}

\begin{proof}
	We first extend $L$ to a finite path. Let $P_1,P_2,\ldots,P_s$ be the connected components of $L$. Since $L$ is a finite linear forest, each $P_i$ is a finite path, possibly consisting of a single vertex.
	
	We join these paths one at a time. Suppose that $P_1,\ldots,P_i$ have already been joined into a finite path $Q_i$, and let $x$ be one of the endpoints of $Q_i$. Choose an endpoint $y$ of $P_{i+1}$, and orient the two paths so that $x$ is the last vertex of $Q_i$ and $y$ is the first vertex of $P_{i+1}$.
	
	If $xy\in E(G)$, we concatenate $Q_i$ and $P_{i+1}$ using the edge $xy$. Otherwise, since both $x$ and $y$ have only finitely many non-neighbours, the set $N_G(x)\cap N_G(y)$ is infinite. We may therefore choose	$
	z\in N_G(x)\cap N_G(y)\setminus\bigl(V(L)\cup V(Q_i)\bigr).
	$
	
	We then concatenate $Q_i$ and $P_{i+1}$ using the path $xzy$. Continuing in this way, we obtain a finite path $P\subseteq G$ such that $L\subseteq P$. If $L$ is empty, we simply take $P$ to consist of one vertex.
	
	Fix an enumeration $V(G)={v_0,v_1,v_2,\ldots}$. We recursively construct an increasing sequence $P=Q_0\subseteq Q_1\subseteq Q_2\subseteq\cdots$ of finite paths. For each $n\in\mathbb{N}$, denote the left and right endpoints of $Q_n$ by $\ell_n$ and $r_n$, respectively. If $Q_n$ consists of a single vertex, we regard this vertex as both endpoints.
	
	Let $v_{m_n}$ be the first vertex in the fixed enumeration which does not belong to $Q_n$. When $n$ is even, we extend $Q_n$ from $r_n$, and when $n$ is odd, we extend it from $\ell_n$. Let $a_n$ denote the endpoint from which $Q_n$ is extended.
	
	If $a_nv_{m_n}\in E(G)$, we add the edge $a_nv_{m_n}$. If $a_nv_{m_n}\notin E(G)$, then $N_G(a_n)\cap N_G(v_{m_n})$ is infinite. Hence we can choose	
	$
	w_n\in N_G(a_n)\cap N_G(v_{m_n})\setminus\bigl(V(Q_n)\cup\{v_{m_n}\}\bigr)
	$
	and extend $Q_n$ by the path $a_nw_nv_{m_n}$. In either case, the resulting graph $Q_{n+1}$ is a finite path properly containing $Q_n$.
	
	Set $R=\bigcup_{n\in\mathbb{N}}Q_n$. We first show that $R$ is spanning. At each step, we add the first vertex of the fixed enumeration which has not yet been included. If some $v_j$ did not belong to $R$, then $m_n\leq j$ for every $n\in\mathbb{N}$. However, the sequence $\langle m_n\rangle_n$ is strictly increasing, because $v_{m_n}$ is added to $Q_{n+1}$ and vertices are never removed. This would give a strictly increasing sequence of natural numbers bounded above by $j$, which is impossible. Therefore $V(R)=V(G)$.
	
	Moreover, we extend the left and right endpoints alternately. Consequently, both sides of $P$ are extended infinitely often, and no endpoint remains in the union. Every vertex of $R$ has degree two in $R$, while $R$ is connected and acyclic. Hence $R$ is a double ray. Since $V(R)=V(G)$, it is a spanning double ray of $G$, and $L\subseteq P\subseteq R$.
\end{proof}

After relabeling the colors if necessary, suppose that the colors appearing in $H$ are $1,2,\ldots,k$. For each $1\leq i\leq k$, let $F_i$ be the linear forest formed by the edges of $H$ having color $i$. Thus $E(H)=\bigcup_{i=1}^kE(F_i)$, and $F_1,F_2,\ldots,F_k$ are pairwise edge-disjoint.

We first construct pairwise edge-disjoint spanning double rays $R_1,R_2,\ldots,R_k$ such that $F_i\subseteq R_i$ for every $1\leq i\leq k$.

Suppose that $R_1,\ldots,R_{i-1}$ have already been constructed, where $1\leq i\leq k$, and define
\[
D_i=\bigcup_{j<i}E(R_j)\cup\bigcup_{i<j\leq k}E(F_j).
\]
Consider the spanning subgraph $G_i=K_\omega-D_i$.

We first observe that $F_i\subseteq G_i$. Indeed, the forests $F_1,\ldots,F_k$ are pairwise edge-disjoint. Moreover, whenever $j<i$, the double ray $R_j$ was constructed after deleting all the forests $F_\ell$ with $\ell>j$, including $F_i$. Consequently, $E(R_j)\cap E(F_i)=\varnothing$.

Every vertex has only finitely many non-neighbours in $G_i$. Indeed, each previously constructed spanning double ray $R_j$ contains exactly two edges incident with any given vertex. Thus the double rays $R_1,\ldots,R_{i-1}$ remove exactly $2(i-1)$ edges incident with each vertex. Moreover, $\bigcup_{i<j\leq k}F_j$ is finite and therefore removes only finitely many additional edges at each vertex.

By Claim \ref{claim1}, there exists a spanning double ray $R_i$ of $G_i$ such that $F_i\subseteq R_i$. Since $R_i\subseteq G_i$, it is edge-disjoint from $R_1,\ldots,R_{i-1}$ and from all the forests $F_{i+1},\ldots,F_k$. This completes the finite induction.

It remains to cover the edges which do not belong to $R_1,\ldots,R_k$. Fix an enumeration $E(K_\omega)=\{e_0,e_1,e_2,\ldots\}$. We recursively construct spanning double rays $R_{k+1},R_{k+2},\ldots$ which are edge-disjoint from all previously constructed double rays.

Suppose that $R_1,\ldots,R_{k+n}$ have already been constructed. Set
\[
A_n=E(K_\omega)\setminus\bigcup_{i=1}^{k+n}E(R_i).
\]
The set $A_n$ is nonempty. Indeed, the union of finitely many spanning double rays has finite degree at every vertex, whereas every vertex of $K_\omega$ has infinite degree.

Let $q_n=\min\{q\in\mathbb{N}:e_q\in A_n\}$, and consider the spanning subgraph
\[
G_n^*=K_\omega-\bigcup_{i=1}^{k+n}E(R_i).
\]
Every vertex has only finitely many non-neighbours in $G_n^*$, because only finitely many spanning double rays have been removed. The graph consisting of the single edge $e_{q_n}$ is a finite linear forest in $G_n^*$. By Claim \ref{claim1}, there exists a spanning double ray $R_{k+n+1}\subseteq G_n^*$ containing $e_{q_n}$.

By construction, the family $\{R_i:i\geq1\}$ is pairwise edge-disjoint. We claim that it covers every edge of $K_\omega$. Suppose, to the contrary, that some edge $e_q$ belongs to none of these double rays. Then $e_q\in A_n$ for every $n\in\mathbb{N}$, and therefore $q_n\leq q$ for every $n\in\mathbb{N}$. On the other hand, the sequence $\langle q_n\rangle_n$ is strictly increasing. Indeed, before stage $n$, all the edges $e_j$ with $j<q_n$ are already covered, by the minimality of $q_n$, and the edge $e_{q_n}$ is covered at stage $n$. Hence $q_{n+1}>q_n$. This gives a strictly increasing sequence of natural numbers bounded above by $q$, a contradiction. Therefore
\[
E(K_\omega)=\bigcup_{i\geq1}E(R_i).
\]

We assign color $i$ to all the edges of $R_i$. This coloring extends the original coloring of $H$, since $F_i\subseteq R_i$ for every $1\leq i\leq k$, and the double rays are pairwise edge-disjoint.

Thus every color class induces a spanning double ray.
\end{proof}

\section{Topological linear arboricity}
\label{sectop}

In this section, we study topological linear arboricity and its relation to linear arboricity. The $F$-limit argument from the previous section is sufficient to preserve the decomposition into linear forest, but it does not by itself control topological circles. Indeed, an $F$-limit of finite linear forests may contain double rays, and if the two tails of such a double ray converge to the same end, its closure in $|G|$ contains a topological circle. Thus, in the topological setting, it is not enough to take an arbitrary sequence of finite approximations.

The infinite square grid in Figure \ref{grid} illustrates why this additional care is necessary. The coloring on the left decomposes the grid into two linear forests, each consisting of double rays. Since the grid has only one end, these double rays give rise to topological circles. Moreover, this cannot be avoided by choosing a different decomposition into two linear forests, as follows from Theorem \ref{oneend}. On the other hand, the coloring on the right decomposes the grid into three topological linear forests: the black components are paths of length $2$, while the blue and red components are rays. Consequently, the square grid has linear arboricity $2$ but topological linear arboricity $3$.

\begin{figure}[H]
	\centering
	\begin{tikzpicture}[
		scale=0.85,
		vertex/.style={circle, fill=black, inner sep=1.7pt},
		blueedge/.style={blue!65, line width=0.8pt},
		rededge/.style={red!80, line width=0.8pt},
		blackedge/.style={black, line width=0.9pt},
		line cap=round
		]
		
		% =========================================================
		% LEFT GRAPH
		% =========================================================
		
		% Horizontal blue lines
		\foreach \y in {0,1,2,3,4,5}{
			\draw[blueedge] (-1,\y) -- (5,\y);
		}
		
		% Vertical red lines
		\foreach \x in {0,1,2,3,4}{
			\draw[rededge] (\x,-1) -- (\x,6);
		}
		
		% Black vertices
		\foreach \x in {0,1,2,3,4}{
			\foreach \y in {0,1,2,3,4,5}{
				\node[vertex] at (\x,\y) {};
			}
		}

		% =========================================================
		% RIGHT GRAPH
		% =========================================================
		
		\begin{scope}[xshift=8cm]
			
			% Horizontal blue lines
			\foreach \y in {0,1,2,3,4,5}{
				\draw[blueedge] (-1,\y) -- (5,\y);
			}
			
			% Vertical red lines
			\foreach \x in {0,1,2,3,4}{
				\draw[rededge] (\x,-1) -- (\x,6);
			}
			
			% ---------------------------------------------------------
			% Black staircase/path
			% ---------------------------------------------------------
			
			% Initial rays
			\draw[blackedge] (0,6) -- (0,5);
			\draw[blackedge] (-1,5) -- (0,5);
			
			% Staircase
			\draw[blackedge] (0,4) -- (1,4);
			\draw[blackedge] (1,5) -- (1,4);
			
			\draw[blackedge] (1,3) -- (2,3);
			\draw[blackedge] (2,4) -- (2,3);
			
			\draw[blackedge] (2,2) -- (3,2);
			\draw[blackedge] (3,3) -- (3,2);
			
			\draw[blackedge] (3,1) -- (4,1);
			\draw[blackedge] (4,2) -- (4,1);
			
			% Black vertices
			\foreach \x in {0,1,2,3,4}{
				\foreach \y in {0,1,2,3,4,5}{
					\node[vertex] at (\x,\y) {};
				}
			}
			
		\end{scope}
		
	\end{tikzpicture}
	
		\caption{Two edge decompositions of the square grid. On the left, the edges are decomposed into two linear forests, both containing double rays. On the right, the edges are decomposed into three topological linear forests.}
	\label{grid}
\end{figure}
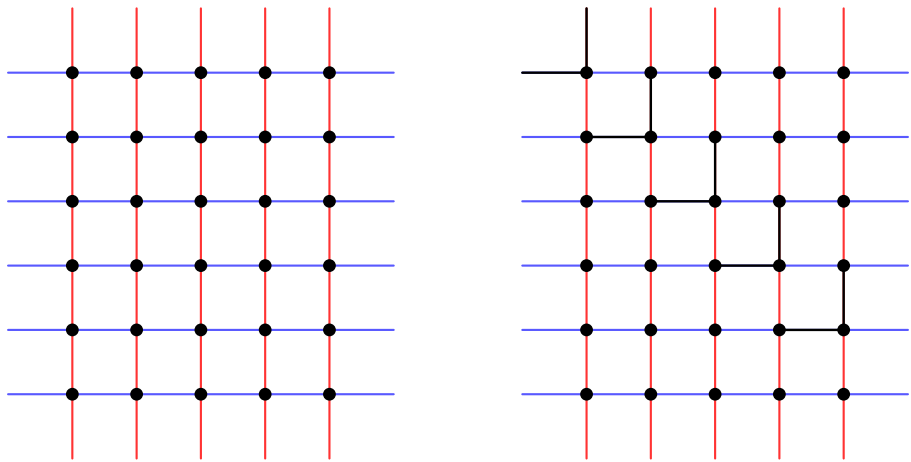

We first record a simple criterion for ruling out topological circles.

\begin{lemma}
	\label{lema1}
	Let $L$ be a linear forest in a graph $G$. If no component of $L$ is a double ray, then $|L|$ contains no topological circle.
\end{lemma}

\begin{proof}
	Suppose, for a contradiction, that $|L|$ contains a topological circle $C$. Choose an edge $e=uv$ of $C$, and let $P$ be the component of $L$ containing $e$.

	Since $L$ is a linear forest, every vertex of $P$ has degree at most two in $P$. We claim that, starting from $e$, the circle $C$ forces $P$ to continue indefinitely in both directions. Indeed, consider one endpoint, say $u$, of $e$. Since $C$ is a topological circle, $u$ cannot be incident with only one edge of $C$. Thus, besides $e$, the circle must contain another edge of $L$ incident with $u$. As $P$ is the component of $L$ containing $e$, this second edge also belongs to $P$.
	
	The same argument may now be repeated at the other endpoint of this second edge. More generally, whenever the circle reaches a vertex $x$ of $P$ along one edge of $P$, it must leave $x$ along another edge of $P$; otherwise $x$ would have degree one in the circle. Since $d_P(x)\leq 2$, this continuation is uniquely determined whenever it exists. Moreover, it cannot terminate at a vertex of $G$, for a terminal vertex would again be incident with only one edge of $C$.
	
	Consequently, starting from $e$ and proceeding through $P$ from $u$, we obtain a ray of $P$. Applying the same argument from the other endpoint $v$ of $e$ yields a ray extending in the opposite direction. These two rays together with $e$ form a double ray. Since $P$ is connected, acyclic, and has maximum degree at most two, it follows that $P$ itself is a double ray.
	
	This contradicts the assumption that no component of $L$ is a double ray. Hence $|L|$ contains no topological circle.
\end{proof}

The next lemma provides a sparse transversal for a countable family of double rays.

\begin{lemma}
	\label{lema2}
Let $G$ be a graph, let $M$ be a matching, and let $\mathcal{D}$ be a countable family of pairwise edge-disjoint double rays. Then there exists a matching $N$ such that every member of $\mathcal{D}$ meets $M\cup N$ and every component of $M\cup N$ has at most three edges.
\end{lemma}

\begin{proof}
Enumerate the members of $\mathcal{D}$ as $R_0,R_1,\ldots$, stopping after finitely many terms if $\mathcal{D}$ is finite. We choose the edges of $N$ recursively. If $R_n$ already meets $M$, no choice is needed. Otherwise, let $X_n$ be the finite set consisting of the endpoints of the edges chosen at earlier stages and their possible mates in $M$. Since $R_n$ is infinite, it contains an edge $e_n$ whose endpoints do not belong to $X_n$. Add $e_n$ to $N$. The construction makes $N$ a matching, and no edge of $M$ can join endpoints of two distinct edges of $N$. Hence each component of $M\cup N$ is a path of length at most three. By construction, every $R_n$ meets $M\cup N$.
\end{proof}

We can now compare the two arboricity parameters.

\begin{theorem}
	\label{count}
Let $G$ be a graph. If $G$ has finite maximum degree $\Delta(G)$, then $\la(G)\leq\latop(G)\leq\la(G)+1$. If $\la(G)$ is infinite, then $\latop(G)=\la(G)$.
\end{theorem}

\begin{proof}
The inequality $\la(G)\leq\latop(G)$ follows directly from Definition \ref{toparb}. Fix a decomposition $\mathcal{F}=\{F_i:i\in I\}$ of $G$ into linear forests with $|I|=\la(G)$. In each connected component $C$ of $G$, consider all double-ray components of the graphs $F_i\cap C$, with $i\in I$. These double rays are pairwise edge-disjoint. Since $G$ has finite maximum degree, every connected component is countable. Lemma \ref{lema2}, applied with $M=\varnothing$, therefore gives a matching $N_C$ meeting every such double ray. Let $N$ be the union of the matchings $N_C$ over all components $C$ of $G$.

For each $i\in I$, every component of $F_i-N$ is a finite path or a ray. Lemma \ref{lema1} implies that $F_i-N$ is a topological linear forest. Hence the family $\{F_i-N:i\in I\}\cup\{N\}$ is an edge decomposition of $G$ into at most $|I|+1$ topological linear forests. This proves $\latop(G)\leq\la(G)+1$ when $\la(G)$ is finite. If $\la(G)$ is infinite, then $|I|+1=|I|$, and therefore $\latop(G)=\la(G)$.
\end{proof}

\begin{corollary}
	\label{cor}
Let $G$ be a graph of finite maximum degree $\Delta (G)$. If Conjecture~1.2 holds, then
\[
\latop(G)\leq\left\lceil\frac{\Delta(G)+1}{2}\right\rceil+1.
\]
\end{corollary}

\begin{proof}
By Theorem \ref{equiv}, Conjecture \ref{conj} implies $\la(G)\leq\lceil(\Delta(G)+1)/2\rceil$. The result now follows from Theorem \ref{count}.
\end{proof}

Thus, assuming Conjecture \ref{conj}, Corollary \ref{cor} yields a general upper bound for topological linear arboricity that exceeds the conjectured finite bound by at most one. In the cubic case, Thomassen proved a stronger bounded-component version of the corresponding finite result.

\begin{theorem}[Thomassen~\cite{thomassen1999}]
	\label{tom}
Every finite cubic graph admits a $2$-edge-coloring in which every monochromatic component is a path of length at most five.
\end{theorem}

The bounded length in Thomassen's theorem is preserved by $F$-limits and yields our first exact computation of topological linear arboricity.

\begin{theorem}
Every cubic graph has topological linear arboricity $2$.
\end{theorem}

\begin{proof}
It suffices to consider a connected graph $G$, which is countable. Enumerate $V(G)=\{v_0,v_1,\ldots\}$, and let $H_n=G[\{v_0,\ldots,v_n\}]$. Complete $H_n$ to a finite cubic graph $G_n$ by the standard doubling construction: starting with $H_n$, repeatedly take two disjoint copies and join corresponding vertices whose current degree is less than $3$. After three steps the resulting graph is cubic and contains an induced copy of $H_n$.

By Theorem \ref{tom}, choose a decomposition of $G_n$ into two linear forests $F_1^n$ and $F_2^n$, every component of which has length at most $5$. For $i\in\{1,2\}$, let $F_i$ be the $F$-limit of $\langle F_i^n\cap G\rangle_n$. The usual finite-intersection argument shows that $F_1$ and $F_2$ are edge-disjoint linear forests covering $E(G)$.

We claim that every component of each $F_i$ has length at most $5$. Otherwise some $F_i$ contains a path $P$ with six edges. For each edge $e\in E(P)$, the set $\{n\in\mathbb{N}:e\in E(F_i^n)\}$ belongs to $F$. Their finite intersection therefore belongs to $F$ and is nonempty, so $P\subseteq F_i^n$ for some $n$, contradicting the choice of $F_i^n$.

Thus neither $F_1$ nor $F_2$ contains a ray or a double ray. By Lemma \ref{lema1}, both are topological linear forests. Since every cubic graph has maximum degree $3$, at least two linear forests are necessary, and hence $\latop(G)=2$.
\end{proof}

We now turn to even-regular graphs. Our first result shows that, for a natural class of such graphs, attaining the degree lower bound for linear arboricity forces the topological linear arboricity to be strictly larger.

\begin{theorem}
	\label{oneend}
Let $G$ be a connected one-ended $2k$-regular graph. If $\la(G)=k$, then $\latop(G)\geq k+1$.
\end{theorem}

\begin{proof}
Let $F_1,\ldots,F_k$ be a decomposition of $G$ into $k$ linear forests. For every vertex $v$, we have
\[
2k=d_G(v)=\sum_{i=1}^k d_{F_i}(v)
\]
and every summand is at most $2$. Hence $d_{F_i}(v)=2$ for every $i$ and every $v$. Consequently, each $F_i$ is a spanning $2$-regular forest, so all its components are double rays. Since $G$ has only one end, the two tails of every such double ray lie in the same end and induce a topological circle. Thus no decomposition into $k$ linear forests is topological, and $\latop(G)\geq k+1$.
\end{proof}

%Together with Theorem~3.4, this gives $\latop(G)=k+1$ whenever $G$ is connected, one-ended, $2k$-regular and $\la(G)=k$. Thus this class necessarily requires one additional color in the topological setting.

Petersen's $2$-factor theorem states that every finite $2k$-regular graph decomposes into $k$ edge-disjoint $2$-factors~\cite{petersen1891}, and the same holds for infinite $2k$-regular graphs by compactness; see~\cite{diestel2025}. In the infinite setting, the structure of these $2$-factors becomes relevant for linear arboricity. If a $2$-factorization can be chosen so that every component of every factor is a double ray, then the factors are themselves linear forests, and the degree lower bound immediately gives $\la(G)=k$. On the other hand, there are infinite regular graphs in which every $2$-factor contains a finite cycle. For such graphs, a $2$-factorization cannot directly yield a decomposition into $k$ linear forests. Figure \ref{fig2} gives such an example already in the $4$-regular case.  The next proposition gives a class of $2k$-regular graphs with linear arboricity greater than $k$, including the graph in Figure \ref{fig2}.

\begin{figure}[H]
	\centering
	\begin{tikzpicture}[scale=.58, every node/.style={circle,fill=black,inner sep=1.7pt}, line cap=round]
		% top finite block
		\coordinate (a) at (-1.4,6.8); \coordinate (b) at (0,6.8); \coordinate (c) at (1.4,6.8);
		\coordinate (d) at (-1.4,4.9); \coordinate (e) at (1.4,4.9);
		\coordinate (f) at (-.45,5.75); \coordinate (g) at (.45,5.75);
		\coordinate (h) at (0,3.75);
		\draw (a)--(b)--(c)--(e)--(d)--(a);
		\draw (a)--(f)--(b); \draw (a)--(g)--(c); \draw (b)--(g); \draw (c)--(f); \draw (d)--(f); \draw (e)--(g);
		\draw (d)--(h)--(e);
		\foreach \p in {a,b,c,d,e,f,g,h} \node at (\p) {};
		
		% infinite chain of crossed squares
		\coordinate (l0) at (-.75,2.55); \coordinate (r0) at (.75,2.55);
		\coordinate (l1) at (-.75,1.45); \coordinate (r1) at (.75,1.45);
		\coordinate (l2) at (-.75,.15); \coordinate (r2) at (.75,.15);
		\coordinate (l3) at (-.75,-.95); \coordinate (r3) at (.75,-.95);
		\coordinate (l4) at (-.75,-2.25); \coordinate (r4) at (.75,-2.25);
		\coordinate (l5) at (-.75,-3.35); \coordinate (r5) at (.75,-3.35);
		\draw (h)--(l0) (h)--(r0);
		\draw (l0)--(r0)--(r1)--(l1)--cycle; \draw (l0)--(r1) (r0)--(l1);
		\draw (l1)--(l2) (r1)--(r2) (l2)--(r2);
		\draw (l2)--(r3) (r2)--(l3) (l2)--(l3) (r2)--(r3) (l3)--(r3);
		\draw (l3)--(l4) (r3)--(r4) (l4)--(r4);
		\draw (l4)--(r5) (r4)--(l5) (l4)--(l5) (r4)--(r5) (l5)--(r5);
		\draw (l5)--++(0,-1.25) (r5)--++(0,-1.25);
		\foreach \p in {l0,r0,l1,r1,l2,r2,l3,r3,l4,r4,l5,r5} \node at (\p) {};
	\end{tikzpicture}
	\caption{A 4-regular graph in which every 2-factor contains a finite cycle.}
	\label{fig2}
\end{figure}
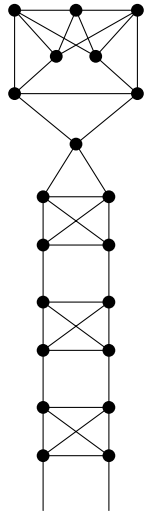

\begin{proposition}
Let $G$ be a $2k$-regular graph. If there is a vertex $v\in V(G)$ such that a component of $G-v$ is finite, then every $2$-factor of $G$ contains a finite cycle. In particular, $\la(G)\geq k+1$.
\end{proposition}

\begin{proof}
Let $C$ be a finite component of $G-v$, and let $H$ be a $2$-factor of $G$. Since every edge leaving $C$ has its other endpoint at $v$, the cut $\delta_H(C)$ has size at most two. Moreover, $|\delta_H(C)|$ has the same parity as $\sum_{x\in C}d_H(x)=2|C|$, so $|\delta_H(C)|$ is either zero or two. If it is zero, then $H[C]$ is a finite $2$-regular graph and therefore contains a finite cycle. If it is two, then either $H[C]$ already contains a cycle or its nontrivial component is a path joining the two vertices incident with $\delta_H(C)$; in the latter case this path, together with the two cut edges and $v$, is a finite cycle of $H$.

If $\la(G)=k$, then in every decomposition of the $2k$-regular graph $G$ into $k$ linear forests, every vertex has degree two in every color. Each color class would therefore be a $2$-factor, contradicting the first part. Hence $\la(G)\geq k+1$.
\end{proof}

We now turn to $2k$-regular graphs of large girth. The role of the girth condition is that it allows the finite cycles occurring in a $2$-factorization to be simultaneously broken by a matching. Our goal is to show that, under the same girth assumption, the upper bound known in the finite setting remains valid for topological linear arboricity. The finite result behind this argument is due to Mishra.

\begin{theorem}[Mishra~\cite{mishra2025}]
	Let $G$ be a finite $2k$-regular graph with girth $g(G)\geq2k$. Then $\la(G)\leq k+1$.
\end{theorem}

For the infinite case, we need more than the numerical conclusion of this theorem. The proof in~\cite{mishra2025} yields the following factorization together with a matching meeting all of its finite cycles.

\begin{lemma}[\cite{mishra2025}]
	\label{lem}
	Let $H$ be a finite $2k$-regular graph of girth at least $2k$. Then there are pairwise edge-disjoint $2$-factors $Q_1,\ldots,Q_k$ covering $E(H)$ and a matching $M$ that meets every cycle of every $Q_i$.
\end{lemma}

The preceding lemma provides the structure that we will pass to the infinite setting by means of $F$-limits. For this purpose, we need finite $2k$-regular graphs containing prescribed finite subgraphs of the original graph while preserving the same lower bound on the girth. This is guaranteed by the following completion result of Mishra.

\begin{lemma}[Mishra~\cite{mishra2025}]
	\label{lem2}
	Let $H$ be a finite graph with $\Delta(H)\leq d$ and girth at least $g$. Then $H$ is contained as an induced subgraph in a finite $d$-regular graph of girth at least $g$.
\end{lemma}

\begin{theorem}
	\label{girth}
Let $G$ be a $2k$-regular graph with girth $g(G)\geq2k$. Then $\latop(G)\leq k+1$.
\end{theorem}

\begin{proof}
It is enough to work in one connected component, which is countable. Enumerate its vertices as $v_0,v_1,\ldots$, and let $H_n=G[\{v_0,\ldots,v_n\}]$. By Lemma \ref{lem2}, each $H_n$ is an induced subgraph of a finite $2k$-regular graph $G_n$ of girth at least $2k$.

By Lemma \ref{lem}, choose pairwise edge-disjoint $2$-factors $Q_1^n,\ldots,Q_k^n$ covering $E(G_n)$ and a matching $M_n$ meeting every cycle of every $Q_i^n$. For $1\leq i\leq k$, let $F_i$ be the $F$-limit of $\langle Q_i^n\cap G\rangle_n$, and let $M$ be the $F$-limit of $\langle M_n\cap G\rangle_n$.

The usual ultrafilter argument shows that $F_1,\ldots,F_k$ are pairwise edge-disjoint and cover $E(G)$. We verify that each $F_i$ is a $2$-factor. Fix $v\in V(G)$. For all sufficiently large $n$, the graph $H_n$ contains $v$ and all its neighbors in $G$. Since $H_n$ is induced, $v$ already has degree $2k$ in $H_n$, and therefore no edge added in the completion $G_n$ is incident with $v$. Thus $Q_i^n$ uses exactly two edges of $G$ at $v$. Since only finitely many pairs of edges are incident with $v$, there is a pair ${e,f}$ such that the set of indices $n$ for which ${e,f}\subseteq E(Q_i^n)$ belongs to $F$. Hence $e,f\in E(F_i)$, and therefore $d_{F_i}(v)=2$.

The $F$-limit $M$ is a matching. We claim that it meets every finite cycle contained in any $F_i$. Let $C\subseteq F_i$ be a finite cycle. The set of indices for which $C\subseteq Q_i^n$ belongs to $F$. For each such index, $M_n$ contains an edge of $C$. Since $C$ has finitely many edges, one edge of $C$ belongs to $M$.

Let $\mathcal{D}$ be the family of double-ray components of the factors $F_i$ that are not met by $M$. This is a countable family of pairwise edge-disjoint double rays. By Lemma \ref{lema2}, there is a matching $N$ such that every member of $\mathcal{D}$ meets $M\cup N$ and every component of $L=M\cup N$ has at most three edges. Hence $L$ is a topological linear forest.

For $1\leq i\leq k$, let $L_i=F_i-L$. Every finite cycle of $F_i$ meets $M\subseteq L$, and every double-ray component of $F_i$ meets $L$. Thus $L_i$ is a linear forest with no double-ray component, and Lemma \ref{lema1} implies that it is topological. The graphs $L_1,\ldots,L_k,L$ are pairwise edge-disjoint and their edge sets cover $E(G)$. Therefore $\latop(G)\leq k+1$.
\end{proof}

\begin{corollary}
Let $G$ be a graph with $\Delta(G)=2k$ and girth $g(G)\geq\Delta(G)$. Then $\latop(G)\leq k+1$.
\end{corollary}

\begin{proof}
It is enough to work in one connected component. Enumerate its vertices as $V(G)=\{v_0,v_1,\ldots\}$, let $H_n=G[\{v_0,\ldots,v_n\}]$, and use Lemma \ref{lem2} to embed $H_n$ as an induced subgraph of a finite $2k$-regular graph $G_n$ of girth at least $2k$. Apply Lemma \ref{lem} in every $G_n$ and take $F$-limits after restricting the resulting factors and matchings to $G$. The proof of Theorem~\ref{girth}  gives an edge decomposition into at most $k+1$ topological linear forests.
\end{proof}

We conclude with two questions suggested by the results of this paper. The first concerns the optimal general upper bound for topological linear arboricity.

\begin{question}
	\label{ques1}
Does every graph $G$ with finite maximum degree $\Delta(G)$ satisfy
$$
\left\lceil\frac{\Delta(G)}{2}\right\rceil
\leq \latop(G)
\leq
\left\lceil\frac{\Delta(G)+1}{2}\right\rceil?
$$

\end{question}

The Question \ref{ques1} is natural in view of the bounds obtained above. For infinite cubic graphs, as well as for infinite $2k$-regular graphs of girth at least $2k$, the upper bounds for topological linear arboricity coincide with the corresponding finite bounds. It is therefore natural to ask whether the same upper bound holds for every graph of finite maximum degree.

Our results also show that connected infinite cubic graphs have linear arboricity $2$, and hence admit an edge decomposition into two linear forests. There is, however, no reason for the two forests obtained in such a decomposition to be isomorphic. This suggests asking whether the decomposition can be chosen with this additional symmetry. In the finite setting, Wormald asked whether every connected cubic graph whose number of vertices is divisible by $4$ admits an edge decomposition into two isomorphic linear forests~\cite[Problem~13]{wormald1987}. This leads naturally to the following infinite analogue.

\begin{question}
	Does every connected infinite cubic graph admit an edge decomposition into two isomorphic linear forests?
\end{question}

The passage from the finite problem to the infinite one is less immediate than for linear arboricity conjecture. In particular, the $F$-limit construction used earlier does not automatically preserve isomorphism between the two limiting forests, since isomorphic forests at each finite stage may have non-isomorphic $F$-limits. It would therefore be interesting to determine whether the finite decompositions can be chosen with enough additional structure to retain isomorphism in the limit. Related finite versions of the isomorphic-decomposition problem were studied in~\cite{fouquet2009,kronenberg2022}.

\section*{Acknowledgments}
The first and third authors were supported by Fundação de Amparo à Pesquisa do Estado de São Paulo (FAPESP), under grant numbers 2025/12199-3 and 2024/01755-0, respectively. The second author was supported by Coordenação de Aperfeiçoamento de Pessoal de Nível Superior (CAPES), Finance Code 001.

\bibliographystyle{plain}
\bibliography{arboricity}

\Addresses
\end{document}